\documentclass{ifacconf}

\usepackage{graphicx}      
\usepackage{natbib}        
\usepackage{amsmath} 
\usepackage{amssymb}  

\usepackage{xcolor}
\usepackage{tikz}
\usepackage{tkz-graph}  
\usepackage[american]{circuitikz}  
\usepackage{algorithm}
\usepackage{caption}  

\graphicspath{{code_figs/}}

\newtheorem{remark}{Remark}
\newtheorem{proposition}{Proposition}
\newtheorem{lemma}{Lemma}

\newcommand{\figref}[1]{Fig. \ref{#1}}

\newcommand{\cK}{\mathcal{K}}
\newcommand{\cC}{\mathcal{C}}
\newcommand{\cKH}{\cK_\mathcal{H}}
\newcommand{\cKR}{\cK_\mathcal{R}}
\newcommand{\cKRI}{\cKR(R_{inner})}
\newcommand{\cKHO}{\cKH(H_{outer})}
\newcommand{\cA}{\mathcal{A}}
\newcommand{\minuszero}{\setminus \{0\}}
\newcommand{\dual}[1]{#1^*}
\newcommand{\J}[1]{\partial f(#1)}
\newcommand{\innerproduct}[2]{#1^T #2}
\newcommand{\outerproduct}[2]{(#1 #2^T)}
\newcommand{\norm}[2]{||#1||_{#2}}

\newcommand{\posvector}{\ensuremath{p}}
\newcommand{\posmatrix}{\ensuremath{P}}
\newcommand{\primalvector}{\ensuremath{r}}
\newcommand{\dualvector}{\ensuremath{h}}
\newcommand{\domright}{\ensuremath{\bar{r}}}
\newcommand{\domleft}{\ensuremath{\bar{h}}}

\newcommand{\widening}{\ensuremath{w}}
\newcommand{\widenedelement}{\ensuremath{{\bar{A}}}}
\newcommand{\widenedsetgeneral}[1]{\ensuremath{\bar{\cA}_{#1}}}
\newcommand{\widenedset}{\ensuremath{\widenedsetgeneral{\widening}}}
\newcommand{\numrays}{m}

\newcommand{\parameters}{c}
\newcommand{\designmatrix}{U}
\newcommand{\numnonfixed}{l}

\newcommand{\electricalvaluescommon}{\ensuremath{L_1 = 30, C_2 = 10, R_2 = 5, C_3 = 1}}
\newcommand{\electricalnonlineardefault}{\ensuremath{f(x_3) = x_3 - 1.5\tanh(x_3)}}
\newcommand{\electricalnonlinearbounds}{\ensuremath{f'_l = -0.5 \text{ and } f'_u = 1}}
\newcommand{\electricalresistanceosc}{1}
\newcommand{\electricalresistancebi}{50}
\newcommand{\electricaltolerance}{\ensuremath{10\%}}

\newcommand{\electricalnewtolerance}{\ensuremath{20\%}}
\newcommand{\electricalnewresistancebi}{61}

\newcommand{\consensusanalnumraysold}{52}
\newcommand{\consensusanalnumraysnew}{7}
\newcommand{\consensussynampl}{2}
\newcommand{\consensussyngain}{4.8}

\newcommand{\springdamperanalmin}{1}
\newcommand{\springdamperanalmax}{3}

\newcommand{\springdampersynmin}{-3}
\newcommand{\springdampersynmax}{3}
\newcommand{\springdampersyninitialkp}{3.02}

\DeclareMathOperator{\interior}{int}
\DeclareMathOperator{\conichull}{conic-hull}
\DeclareMathOperator{\convexhull}{convex-hull}
\DeclareMathOperator{\offdiag}{off-diag}

\newenvironment{proof}{
   \begin{pf}
}{
   \hfill $\blacksquare$
   \end{pf}
}

\begin{document}
\begin{frontmatter}

\title{
An Optimization Approach to Verifying and Synthesizing $\cK$-cooperative Systems%
\thanksref{footnoteinfo}} 

\thanks[footnoteinfo]{D. Kousoulidis is supported by the Engineering and Physical Sciences Research Council (EPSRC) of the United Kingdom.}

\author[First]{Dimitris Kousoulidis} 
\author[First]{Fulvio Forni}

\address[First]{Department of Engineering, University of Cambridge, UK \\
   (dk483@eng.cam.ac.uk,
   f.forni@eng.cam.ac.uk)}

\begin{abstract}                
Differential positivity and $\cK$-cooperativity, a special case of differential positivity, extend 
differential approaches to control to nonlinear systems with multiple equilibria,
such as switches or multi-agent consensus.
To apply this theory,
we reframe conditions for strict $\cK$-cooperativity as an optimization problem.
Geometrically, the conditions correspond to finding a cone that a set of linear operators leave invariant.
Even though solving the optimization problem is hard,
we combine the optimization perspective with the geometric intuition
to construct a heuristic
cone-finding algorithm centered around Linear Programming (LP).
The algorithm we obtain is unique in that it
modifies existing rays of a candidate cone instead of adding new ones.
This enables us to
also take a first step in tackling the synthesis problem
for $\cK$-cooperative systems.
We demonstrate our approach on some examples,
including one in which we repurpose our algorithm to obtain
a novel alternative tool for computing polyhedral Lyapunov functions
of bounded complexity.
\end{abstract}


\end{frontmatter}


\section{Introduction}
Multistable systems are difficult to analyze with classical system-theoretic methods.
The presence of several fixed points limits the use of fundamental methods like Lyapunov theory,
constraining its use within the neighborhood of each
attractor.
Feedback control design is even more challenging.
Even the task of tuning the parameters of a simple bistable switch
for performance or robustness to perturbations
pushes any classical tool of nonlinear control to its limits.
In this paper we begin to address these
issues by proposing a robust, tractable approach to nonlinear analysis and feedback design for 
systems whose attractor landscape is characterized by the presence of multiple fixed points. 
The approach builds on differential positivity (\cite{forni_differentially_2016}),
which studies systems by looking at the positivity of the linearized dynamics (along any trajectory of the system).
Tractability follows from a two-step optimization iteration based on Linear Programming,
which provides a novel numerical tool for analysis and feedback design for monostable and multistable closed-loop systems. 

Differential positivity extends differential analysis (\cite{forni_differential_2014-1}).
The approach is similar to contraction theory
(\cite{lohmiller_contraction_1998-1, pavlov_uniform_2006, russo_global_2010, forni_differential_2014}), 
which characterizes the stability/contraction of a system from the stability of its linearizations
along any possible system trajectory.
This makes contraction a powerful approach for problems where the fixed point varies with parameters.
Similarly to how contraction theory links the convergence of system trajectories to the stability of the linearizations,
differential positivity links the behavior of the nonlinear system (monotonicity, multistability, etc.) to the positivity of its linearizations.

A linear system is positive if its trajectories contract a cone.
This is a form of projective contraction that leads to the Perron-Frobenius theorem,
which entails the existence of a slow dominant mode in the system dynamics
and is also related to the presence of a dominant eigenvector within the cone (\cite{bushell_hilberts_1973}).
In a similar way, a system is differentially positive if its linearized dynamics contract a cone field
(a cone that depends on the state of the system).
When the cone field is constant and the system state belongs to a vector space,
a differentially positive system is a monotone system (\cite{hirsch_chapter_2006,angeli_monotone_2003}),
therefore its trajectories preserve a partial order relation on the
system state space.
Furthermore, almost all bounded trajectories of the system converge to some fixed point, asymptotically. 
We call these systems \emph{$\cK$-cooperative}, to emphasize the role of the constant cone field $\cK$ and important
connections with the literature (\cite{hirsch_chapter_2006}).

Finding the contracting cone of a $\cK$-cooperative system is provably difficult (\cite{protasov_when_2010}).
For this reason, most of the literature assumes that a suitable cone is readily available (\cite{hirsch_chapter_2006,angeli_monotone_2003})). 
In contrast, in this paper we provide a heuristic algorithm to attempt to find such cones.
The theory builds on a previous attempt (\cite{kousoulidis_finding_2019}).
The key difference from our previous algorithm
is that each iteration of the new algorithm modifies existing rays instead of adding new ones.
This has computational advantages and allows us to find
cones of lower complexity (fewer extreme rays).
We also make use of this to attempt to synthesize $\cK$-cooperativity,
enabling a new approach
to robust design for multistable systems.

Our algorithm is
based on optimization and Linear Programming.
To make use of optimization techniques, we introduce
a quantification of how far a given system is from
being $\cK$-cooperative with respect to a given cone.
This is then used
to iteratively reshape the cone to improve this distance
until we obtain a contracting cone.
We illustrate our algorithm in practice in Section \ref{sec:examples} by analyzing and designing
a robust electrical switch.
We also investigate a nonlinear consensus problem with repulsive interactions.
Finally, through a suitable system augmentation,
in our third example we show how to use our algorithm 
to build polyhedral Lyapunov functions for stability analysis.

The first two sections below introduce the language of the paper. 
The main cone finding algorithm is presented in 
Sections \ref{sec:sys_anal} and \ref{sec:control_design} for analysis and synthesis respectively.
This is followed by the examples.

{
   \small
   \vspace{4mm}
   \textbf{Notation:}
   When used on matrices or vectors, inequalities are always meant in the element-wise sense.
   Combined with inequalities and used on a matrix, $\offdiag(\cdot)$ implies that the inequalities apply only to the off-diagonal entries of the matrix.
   The interior of a set $\mathcal{S}$ is denoted by $\interior(\mathcal{S})$.
   A proper cone is a set $\cK$ such that:
   (i) if $\primalvector_1,\, \primalvector_2 \in \cK \text{ and }0 \leq \posvector_1,\,\posvector_2 \in \mathbb{R}$, then $\posvector_1 \primalvector_1 + \posvector_2 \primalvector_2 \in \cK$;
   (ii) $\interior(\cK) \neq \{\}$;
   and (iii) if $\primalvector \in \cK \minuszero$, then $-\primalvector \notin \cK$.
   The dual cone of a cone $\cK$ is denoted by $\dual{\cK}$ and defined as $\{\dualvector: \innerproduct{\dualvector}{\primalvector} \geq 0, \, \forall \primalvector \in \cK\}$.
   The dual cone can also be interpreted geometrically as
   the set of all half-spaces that contain the cone.
   Considering two cones $\cK_1$ and $\cK_2$, $\cK_1 \subseteq \cK_2$ denotes the usual set inclusion. We use $\cK_1 \subset \cK_2$ to denote $\primalvector \in \cK_1 \minuszero \implies \primalvector \in \interior(\cK_2)$.
   When $\mathcal{A}$ denotes a finite set,
   we use $|\mathcal{A}|$ to denote its index set.
}

\section{$\cK$-cooperativity} \label{sec:background}

\subsection{Definitions and Properties}

We say that a continuous time linear system $\dot{x} = Ax, \, x \in \mathbb{R}^n,$ leaves a cone $\cK$ invariant if
$$x(0) \in \cK \implies x(t) \in \cK \text{ for all } t \geq 0$$
This is equivalent to the \emph{sub-tangentiality} condition
\cite[Theorem 3.11]{berman_nonnegative_1989}:
\begin{equation}
   x \in \cK,\, y \in \dual{\cK},\, \innerproduct{y}{x} = 0 \implies \innerproduct{y}{Ax} \geq 0
   \label{eq:subtang}
\end{equation}

Similarly, we say that the system contracts a cone $\cK$ if
$$x(0) \in \cK \minuszero \implies x(t) \in \interior(\cK) \text{ for all } t > 0$$
This is implied by the \emph{strict sub-tangentiality} condition
\cite[Theorems 3.7 and 3.26]{berman_nonnegative_1989}:
\begin{equation}
x \in \cK \minuszero,\, y \in \dual{\cK} \minuszero,\, \innerproduct{y}{x} = 0 \implies \innerproduct{y}{Ax} > 0
\label{eq:strict-subtang}
\end{equation}

If we find a cone that the system contracts, we certify that the system is strictly positive.
Strict positivity restricts the qualitative behavior of the system:
the dominant eigenvector is an attractor of the system.

Differential positivity (\cite{forni_differentially_2016}) generalizes the notion of positivity to the nonlinear setting. 
For nonlinear systems, $\dot{x} = f(x)$, this generalization is based on system linearization
and the notion of prolonged system dynamics:
$$ \dot{x} = f(x) \qquad \dot{\delta x} = \J{x} \delta x, $$
where $\J{x}$ is the Jacobian of the vector field at $x \in \mathbb{R}^n$.
Differential positivity is defined with respect to a cone field $\cK(x)$.
In this paper we will study the simpler case of a constant cone field $\cK$.
We say that a nonlinear system differentially contracts a cone $\cK$ if
\begin{equation}
\delta x(0) \in \cK \minuszero \implies \delta x(t) \in \interior(\cK)  \text{ for all } t > 0.
\label{eq:diff-cone-contr}
\end{equation}
In analogy with the linear case, we call these systems \emph{strictly differentially positive}.
Equation
\eqref{eq:diff-cone-contr} is implied by the \emph{strict differential sub-tangentiality} condition 
that requires that, for all $x \in \mathbb{R}^n$
\begin{equation}
   \begin{split}
      \delta x \in \cK \minuszero, \,y \in \dual{\cK} \minuszero,& \\
      \innerproduct{y}{\delta x} = 0 &\implies \innerproduct{y}{\J{x}\delta x} > 0
   \end{split}
\label{eq:diff-strict-subtang}
\end{equation}

Similarly to strict positivity in linear systems, strict differential positivity restricts the qualitative behavior of nonlinear systems:
in our setting, it implies that almost all bounded trajectories converge to some fixed point
\cite[Corollary 5]{forni_differentially_2016}.

The strict differential sub-tangentiality conditions provide a way to verify strict differential positivity.
We call systems that satisfy \eqref{eq:diff-strict-subtang} for some cone \emph{strictly $\cK$-cooperative}.
Since $\eqref{eq:diff-strict-subtang} \implies \eqref{eq:diff-cone-contr}$, strictly $\cK$-cooperative systems are strictly differentially positive systems and share the same convergence result.
For the rest of the paper, we attempt to verify or synthesize strictly $\cK$-cooperative systems.
The main obstacle in doing so is finding a $\cK$ for which \eqref{eq:diff-strict-subtang} holds.

For linear systems, strict positivity has a spectral characterization:
a linear system $\dot{x} = Ax$ is strictly positive
if and only if the rightmost eigenvalue of $A$ is simple
and real (\cite{vandergraft_spectral_1968}).
Unfortunately,
for a nonlinear system $\dot{x} = f(x)$,
this does not generalize to a sufficient
condition on the eigenvalues of the Jacobians $\J{x}$.
The presence of a rightmost eigenvalue in the Jacobian at every $x$ is only a \emph{necessary} condition.
This is analogous to how we cannot, generally, establish the convergence of all trajectories of a system to a unique fixed point from the analysis of
the Jacobian eigenvalues
(see related counter-examples in \cite[Example 4.22]{khalil_nonlinear_2002} and \cite{leonov_problems_2010}).

\section{General Formulation}

\subsection{Conical Relaxation and Robustness} \label{subsec:conical_relax}
To verify \eqref{eq:diff-strict-subtang} with finite computations
we relax the set of all Jacobians, $\{\J{x}\}$, to a \emph{finitely} generated set of matrices, $\cA$, such that for each $x \in \mathbb{R}^n$,
\begin{equation}
   \J{x} \in \conichull(\cA) \minuszero, \label{eq:conical_relax}
\end{equation}
where $\conichull(\cA)$ is defined as 
$$\conichull(\cA) = \left\{ A : A = \sum_{i=1}^k \posvector_i A_i , \,\posvector_i \geq 0\right\}$$

We refer to a set $\cA$ that satisfies \eqref{eq:conical_relax} as a \emph{conical relaxation} of $\{\J{x}\}$.
The use of a conical relaxation to verify \eqref{eq:diff-strict-subtang}
follows from the next Lemma.
\begin{lemma}[Convexity of Sub-tangentiality]
\label{lemma:conv_sub-tang}
If $\cA$ is a conical relaxation of $\{\J{x}\}$, 
and each $A_i \in \cA$ satisfies \eqref{eq:strict-subtang} with respect to the same cone $\cK$,
then $f(x)$ satisfies \eqref{eq:diff-strict-subtang}. 
\end{lemma}
\begin{proof}
Assuming left-hand side of \eqref{eq:diff-strict-subtang},
\begin{equation*}
   \innerproduct{y}{\J{x}\delta x} = \innerproduct{y}{\left(\sum_{i=1}^k \posvector_i^x A_i\right)\delta x} = \sum_{i=1}^k \posvector_i^x (\innerproduct{y}{A_i\delta x}) > 0, 
\end{equation*}
where the strict inequality follows from \eqref{eq:strict-subtang}, using the fact that at least one $\posvector_i^x >0$ (since $\J{x} \neq 0$). 
\end{proof}

We can also use conical relaxations to analyze robustness to uncertainty in the dynamics:
\begin{itemize}
   \item We get robustness to some perturbations `for free' since our certificate also holds for any $\dot{x} = f(x) + g(x)$ where $\partial g(x) \in \conichull(\cA)$.
   \item We can adapt a conical relaxation to incorporate specific perturbations.
   If our perturbed system can be represented in the linearizations by $\J{x} + \mathcal{Q}$,
   where $\mathcal{Q}$ is a given family of perturbations, we can use a new conical relaxation $\cA_\mathcal{Q}$ that satisfies $\J{x} + \mathcal{Q} \in \conichull(\cA_\mathcal{Q}) \minuszero$.
\end{itemize}

We therefore treat the problem of finding cones for system analysis as the problem of finding a common cone that is contracted by (all elements of) a given $\cA$.

\begin{remark}
Naively producing a tight conical relaxation for a general system can lead to a combinatorial explosion in the number of elements in $\cA$.
We do not presently focus on the general construction of conical relaxations and instead manually derive them for the applications considered.
\end{remark}

\subsection{Cone Representations}
In our search for cones, we limit ourselves to the family of \emph{polyhedral} cones.
Polyhedral cones can be represented in two ways:
\begin{description}
	\item[R-representation] given matrix $R$,
	\begin{equation}
	\cKR(R) = \{\primalvector: \primalvector = R\posvector, \posvector\geq 0\} \label{eq:r-rep} \
	\end{equation}
	\item[H-representation] given matrix $H$, 
	\begin{equation}
	\cKH(H) = \{\primalvector: H\primalvector \geq 0 \} \label{eq:h-rep} \
	\end{equation}
\end{description}
Assuming the cones are proper, strict inequalities characterize their interior.
We focus on verification and synthesis of $\cK$-cooperative systems with respect to
R-representation cones, noting that the analysis and algorithms can be extended to H-representation cones through duality.
Specifically, if we introduce $\cA^* = \{A^T: A \in \cA\}$,
then it can be shown that using $\cA$ to verify $\cK$-cooperativity with respect to $\cKH(H)$
yields the same conditions as
using $\cA^*$ and $\cKR(H^T)$.

\subsection{Feasibility Formulation}
We reformulate \eqref{eq:subtang} and \eqref{eq:strict-subtang} for R-representation polyhedral cones:
\begin{lemma}
   \label{lemma:coni_r_test}
   For a given matrix $A$ and proper cone $\cKR(R)$,
   $A$ satisfies \eqref{eq:subtang} for $\cKR(R)$
   if (and only if) there exists a matrix $P$ with nonnegative off-diagonal entries ($\offdiag(P) \geq 0$) such that:
   \begin{equation}
      AR = RP \label{eq:main_condition}
   \end{equation}
   Furthermore, 
   $A$ also satisfies \eqref{eq:strict-subtang} for $\cKR(R)$
   if (and only if) there exists a $P$ with $\offdiag(P) > 0$.
\end{lemma}
\begin{proof}
   If: Let $\primalvector \in \cK \minuszero,\, \dualvector \in \dual{\cK} \minuszero,\, \innerproduct{\dualvector}{\primalvector} = 0$, and $\cKR(R)$ is proper.
  
   Then, for any $\alpha \in \mathbb{R}$, $\innerproduct{\dualvector}{A \primalvector} = \innerproduct{\dualvector}{(A +\alpha I)\primalvector} = \innerproduct{\dualvector}{(A+\alpha I)R\posvector} = \innerproduct{\dualvector}{R(P+\alpha I)\posvector}$,
   where the second identity follows from the additional identity $r = Rp$, for some $p \geq 0$. 
   
   From our assumptions we have $0 \neq \posvector \geq 0$ and $0 \neq (\dualvector^T R) \geq 0$.
   As such we satisfy \eqref{eq:subtang} for $A$ if $(P+\alpha I) \geq 0$.
   Since we can choose $\alpha$ freely, this is equivalent to $\offdiag(P) \geq 0$. 
   Similarly, we satisfy \eqref{eq:strict-subtang} for $A$ if $\offdiag(P) > 0$.

   Only if: See \cite[Theorems 3.3.9 and 3.3.41]{berman_nonnegative_1989}.
\end{proof}

Putting everything together:
\begin{proposition}[Feasibility Formulation]
\label{prop:feas}
A nonlinear system $\dot{x} = f(x)$ is
strictly $\cK$-cooperative if, given an $\cA$ that satisfies \eqref{eq:conical_relax},
there exists a matrix $R$ satisfying:
\begin{itemize}
   \item $\cKR(R)$ is a proper cone
   \item For each $A_i \in \cA$, there exists a matrix $\posmatrix_i$ such that $A_i R = R\posmatrix_i$ and $\offdiag(\posmatrix_i) > 0$
\end{itemize}
\end{proposition}
\begin{proof}
Implied by combining Lemmas \ref{lemma:conv_sub-tang} and \ref{lemma:coni_r_test}.
\end{proof}

If $R$ is left as a free parameter this test is intractable.
On the other hand, when $R$ is
fixed the test can be carried out using 
Linear Programming (LP). 
However, this LP test only
gives a `Yes' or `No' answer.
For our goal of \emph{finding} a suitable $R$
we are interested in
introducing a measure of how far a given $R$ is from satisfying Proposition \ref{prop:feas}.

\subsection{Measuring Distance to $\cK$-cooperativity}
We want to quantify how far a given set of positive matrices is from contracting a given candidate cone.
This measure should be independent of the representation of the cone and 
compatible with LP.
To this end,
we introduce the set of widened operators \widenedset, where for each $\widenedelement_i \in \widenedset$:
\begin{equation}
   \widenedelement_i = A_i + \widening_i \outerproduct{\primalvector_i}{\dualvector_i} \quad (A_i \in \cA)
\end{equation}
where $\widening_i$ is a scalar `widening coefficient', $\primalvector_i$ is some fixed vector in $\interior(\cK)$, and $\dualvector_i$ is some fixed vector in $\interior(\dual{\cK})$
(we explore some strategies for selecting these later).
The rank one matrix $\outerproduct{\primalvector_i}{\dualvector_i}$
projects all points $x$ to $(\innerproduct{\dualvector_i}{x})\primalvector_i$.

\begin{lemma}
   \label{lemma:finite_widening}
   For any given matrix $A$, cone $\cK$, $\primalvector \in \interior(\cK)$, and $\dualvector \in \interior(\dual{\cK})$, there will exist some finite $\widening$
   such that $[A + \widening\outerproduct{\primalvector}{\dualvector}]$ satisfies \eqref{eq:subtang} with respect to $\cK$ 
\end{lemma}
\begin{proof}
   We are only concerned with the sign of $\innerproduct{y}{[A + \widening\outerproduct{\primalvector}{\dualvector}]x}$, so we fix $||y||_2 = ||x||_2 = 1$
   (the implication is trivial when $y$ or $x$ are 0).
   Because $\primalvector \in \interior(\cK)$, and $\dualvector \in \interior(\dual{\cK})$,
   $\innerproduct{y}{\primalvector} > 0$ and $\innerproduct{\dualvector}{x} > 0$.
   We denote $\inf_{y}(\innerproduct{y}{\primalvector}) = \alpha > 0$ and $\inf_{x}(\innerproduct{\dualvector}{x}) = \beta > 0$.

   Then:
   \resizebox{.9\hsize}{!}{
   $
   \innerproduct{y}{[A + \widening\outerproduct{\primalvector}{\dualvector}]x} =
   \innerproduct{y}{Ax} + \widening(\innerproduct{y}{\primalvector})(\innerproduct{\dualvector}{x}) \geq -||A||_2 + \widening \alpha \beta
   $
   } 
   
   As such, we can guarantee that $\innerproduct{y}{[A + \widening\outerproduct{\primalvector}{\dualvector}]x} \geq 0$ if $\widening \geq ||A||_2/(\alpha \beta)$.
\end{proof}

\begin{lemma}
   \label{lemma:strict_widening}
   Given $[A + \widening\outerproduct{\primalvector}{\dualvector}]$ and $\cK$ that satisfy \eqref{eq:subtang} with $\primalvector \in \interior(\cK)$ and $\dualvector \in \interior(\dual{\cK})$, $[A + \widening^*\outerproduct{\primalvector}{\dualvector}]$ will satisfy \eqref{eq:strict-subtang} for all $w^* > w$ with the same $\cK$. 
\end{lemma}
\begin{proof}
   Since $[A + \widening\outerproduct{\primalvector}{\dualvector}]$ satisfy \eqref{eq:subtang},
   $\innerproduct{y}{[A + \widening\outerproduct{\primalvector}{\dualvector}]x} \geq 0$.
   Also, since $\primalvector \in \interior(\cK)$, and $\dualvector \in \interior(\dual{\cK})$,
   $\innerproduct{y}{\primalvector} > 0$ and $\innerproduct{\dualvector}{x} > 0$.
   Letting $\bar{\widening} = \widening^*-\widening > 0$,
   $\innerproduct{y}{[A + \widening^*\outerproduct{\primalvector}{\dualvector}]x} =  
   \innerproduct{y}{[A + \widening\outerproduct{\primalvector}{\dualvector}]x} + \bar{\widening}(\innerproduct{y}{\primalvector})(\innerproduct{\dualvector}{x}) > 0$
\end{proof}

Lemmas \ref{lemma:finite_widening} and \ref{lemma:strict_widening} mean that we can use $\{\widening_i\}$ to quantitatively measure how close we are to a contracting cone:
if we leave $\{\widening_i\}$ as free parameters, the tests in Lemma \ref{lemma:coni_r_test} will always be feasible for $\widenedelement_i \in \widenedset$ and a fixed $\cKR(R)$,
and, if we find $\{\widening_i\} < 0$, we can conclude strict $\cK$-cooperativity.

To summarize:
\begin{proposition}[Optimization Formulation]
   \label{prop:opt}
Consider the following optimization problem:
\begin{alignat}{2}
   &\!\min_{\{\widening_i,\posmatrix_i\}}\,&\quad& \widening \label{eq:opt} \\
   &\text{subject to, $\forall\,i \in |\cA|$:}& & \widening \geq \widening_i, \nonumber \\
   & & &\offdiag(\posmatrix_i) \geq 0,\,  \nonumber \\
   & & & \left[A_i + \widening_i \outerproduct{\primalvector_i}{\dualvector_i}\right]R = R\posmatrix_i \nonumber
\end{alignat}
Where $|\cA|$ is the index set of $\cA$,
$\primalvector_i \in \interior(\cKR(R))$,
and $\dualvector_i \in \interior(\dual{\cKR(R)})$.

If we find a solution with $\widening < 0$,
we can conclude strict $\cK$-cooperativity with respect to the corresponding $\cKR(R)$
for all nonlinear systems $\dot{x} = f(x)$
that satisfy \eqref{eq:conical_relax} for the $\cA$ used.
\end{proposition}
\begin{proof}
Implied by combining Lemmas 
\ref{lemma:conv_sub-tang}-\ref{lemma:strict_widening}.
\end{proof}

We can now make use of $\widening$ to iteratively refine an initial candidate cone $R$, 
searching for a final contracting cone, as described in Section \ref{subsec:ver_algo}.

\section{Verifying $\cK$-cooperativity} \label{sec:sys_anal}
\subsection{Necessary Conditions and Geometric Constraints} \label{subsec:ver_necessary}

In Section \ref{sec:background} we mentioned that 
to verify strict $\cK$-cooperativity we need to find a contracting cone
but that the linearizations directly provide us with some necessary conditions.
We will now use Propositions 4-6 from \cite{kousoulidis_finding_2019} to
elaborate on the necessary conditions and the geometric constraints imposed
by the linearizations.

Firstly, every $A_i \in \cA$ has to be strictly positive, i.e. have a simple, real right-most eigenvalue.
We refer to such eigenvalues as dominant. We also apply the same terminology to their associated right and left eigenvectors, which we denote by $\domright_i$ and $\domleft_i$, and scale such that $\norm{\domleft_i}{1} = 1$ and $\innerproduct{\domleft_i}{\domright_i} = 1$. 
The dominant eigenvectors tell us that all trajectories of $\dot{x} = A_i x$ that start in the open half-space $\innerproduct{\domleft_i}{x(0)} > 0$ will converge (as rays) to the ray $\posvector_0 \domright_i$, for some $0 < \posvector_0 \in \mathbb{R}$.

It follows that for a cone $\cK$ to be contracted by each $A_i$, there must exist orientations of the dominant eigenvectors $\domright_i$ and $\domleft_i$ such that $\domright_i \in \interior(\cK)$ and $\domleft_i \in \interior(\dual{\cK})$.
This must hold for each $A_i \in \cA$.
For this to be possible, there must exist an orientation of the dominant eigenvectors such that $\innerproduct{\domleft_i}{\domright_j} > 0$. 
We can efficiently find such an orientation if it exists or conclude that it does not 
(see Algorithm 1 in \cite{kousoulidis_finding_2019}).
If it does not exist, we can again exclude $\cK$-cooperativity.
If it does exist, it is unique (up to the orientation of one of the eigenvectors).
We then combine our dominant eigenvectors to form two cones, $\cKRI$ and $\cKHO$:
$\cKRI$ is a R-representation polyhedral cone formed from right dominant eigenvectors $\domright_i$, while $\cKHO$ is a H-representation polyhedral cone formed from left dominant eigenvectors $\domleft_i$.
We enforce that $\domright_i \in \interior(\cK)$ and $\domleft_i \in \interior(\dual{\cK})$ by only considering cones $\cK$ that satisfy:
\begin{equation}
   \cKRI \subset \cK \subset \cKHO \label{eq:geo_assump}
\end{equation}

Since $\domright_i \in \interior(\cK)$ and $\domleft_i \in \interior(\dual{\cK})$,
they serve as ideal candidates for $\{\primalvector_i\}$ and $\{\dualvector_i\}$ in Proposition \ref{prop:opt}. 
By using them to produce $\widenedset$,
each new matrix $\widenedelement_i$ will have the same eigenvectors as $A_i$,
will have its dominant eigenvalue shifted by $\widening_i$ to the right,
and will have all other eigenvalues unaffected.

\subsection{The Algorithm}
\label{subsec:ver_algo}

The first step of the verification algorithm is to compute an initial candidate cone.
We want this cone to satisfy \eqref{eq:geo_assump}.
This is because \eqref{eq:geo_assump} is a necessary condition for an invariant cone
and, since we set $\primalvector_i = \domright_i$, $\dualvector_i = \domleft_i$,
\eqref{eq:geo_assump}
also guarantees that we can solve \eqref{eq:opt}
for a finite $\widening$.
Additionally, because we will not be adding new extreme rays
in other steps of the algorithm,
we want to be able to specify the number of extreme rays
$\numrays$ ($\geq n$) as a parameter to our initialization algorithm
(generally, more oscillatory dynamics require cones
with a larger number of extreme rays - \cite{benvenuti_eigenvalue_2004}).

To achieve this, we begin with matrix $R_{inner}$.
Then, until our matrix has $\numrays$ columns, we expand it by appending randomly perturbed columns of $R_{inner}$,
rejecting new vectors that
do not increase the number of extreme rays in our cone or are not inside $\cKHO$.  
To ensure the strict inclusion $\cKRI \subset \cKR(R^{(0)})$, we subtract some small enough positive multiple of
the centroid ray of $\cKRI$ from each column of our matrix.
In this way we now have a matrix $R^{(0)}$
that provides a R-representation of a suitable initial cone.

We then proceed to solve \eqref{eq:opt} with $R = R^{(0)}$,
obtaining values for $\{\widening_i, \posmatrix_i\}$ in the process.
If $\widening < 0$, our system is $\cK$-cooperative and our search is done.
If it is not,
we introduce a new optimization problem that 
approximates 
the effect of changes to $R$ on the solution to \eqref{eq:opt}.
This is done by using the values we have for
$R$ and $\{\widening_i, \posmatrix_i\}$
to linearize the equality constraints of \eqref{eq:opt}
(which are nonlinear when $R$ is treated as a variable)
about the previous solution by taking their derivatives.
The variations of 
$(R, \{\widening_i, \posmatrix_i\})$
about their previous values
are represented by the new variables
$(\delta R, \{\delta \widening_i, \delta \posmatrix_i\})$.
The linearized equality constraints are given by:
$$
[A_i + \widening_i \outerproduct{\domright_i}{\domleft_i}](\delta R)
+ (\delta \widening_i) \outerproduct{\domright_i}{\domleft_i}R
= (\delta R)\posmatrix_i + R(\delta\posmatrix_i)
$$

We also have to ensure that \eqref{eq:geo_assump} holds.
As such, when solving the approximate problem, we need to impose that:
\begin{equation}
\cKRI \subset \cKR(R+\delta R) \subset \cKHO
\label{eq:geo_assump_new}
\end{equation}

The right hand side set inclusion in \eqref{eq:geo_assump_new}
ends up being a linear constraint on $\delta R$ that we can include directly. It can be expressed as
$H_{outer}(R+ \delta R) > 0$.
However, the left-hand side constraint isn't linear.
Instead of also approximating this constraint
via its derivative, we characterize a linear subset
of $\delta R$ for which it
is \emph{guaranteed} to hold.
To guarantee
$\cKRI \subset \cKR(R+\delta R)$,
we first introduce a new matrix $R'$
that satisfies $\cKRI \subset \cKR(R') \subset \cKR(R)$ and has the same dimensions as $R$.
We can obtain such a matrix by, for example,
adding a small enough positive multiple
of the centroid ray of $\cKRI$ to each column of $R$.
We then add a free matrix variable $M$
and the constraints that $(R + \delta R) = R'M$,
$1^T M = 1^T$, 
and $\offdiag(M) \leq 0$, to our optimization problem. 
The constraints imply $-M$ will be Metzler with
$-1$ as its Perron-Frobenius eigenvalue,
and so, all eigenvalues of $M$ will have real part $\geq 1$.
This guarantees that $M$ is an invertible M-matrix ($M^{-1}$ exists and $M^{-1} \geq 0$).
Since $R' = (R + \delta R)M^{-1}$, we have that $\cKR(R') \subseteq \cKR(R+ \delta R)$,
which in-turn guarantees $\cKRI \subset \cKR(R+\delta R)$.

To ensure that $(\delta R, \{\delta \widening_i, \delta \posmatrix_i\})$ are kept small
(so that the linearized equality constraints offer good approximations)
and that the algorithm converges,
we add a penalty and/or a constraint on their norms.
We combine everything to obtain the following LP problem:
\begin{alignat}{2}
   &\!\min_{\delta R,\,M,\,\{\delta\widening_i,\delta\posmatrix_i\}}\,&\>& \bar{\widening} \label{eq:opt_small_delta} \\
   &\text{subject to:}&&
   H_{outer}(R+\delta R) > 0,\, (R+\delta R) = R'M, \nonumber \\
   &&& 1^T M = 1^T,\, \offdiag(M) \geq 0, \nonumber \\
   &&& (\delta R,\,\{\delta\widening_i,\delta\posmatrix_i\}) \text{ small}, \nonumber \\ 
   &\text{and, $\forall\,i \in |\cA|$:}&& \bar{\widening} \geq \widening_i + \delta\widening_i,\, \offdiag(\posmatrix_i + \delta\posmatrix_i) \geq 0, \nonumber \\
   &&& [A_i + \widening_i \outerproduct{\domright_i}{\domleft_i}](\delta R)
   + (\delta \widening_i) \outerproduct{\domright_i}{\domleft_i}R \nonumber \\
   &&& \quad = (\delta R)\posmatrix_i + R(\delta\posmatrix_i) \nonumber 
\end{alignat}

We then update $R$ to $R + \delta R$
and solve the original \eqref{eq:opt} again,
expecting the updated $R$ to be able to solve \eqref{eq:opt}
for a lower $\widening$.
The process is repeated,
stopping when $\widening < 0$ or the algorithm has converged
(which we check by looking at $||\delta R||$).
Algorithm \ref{alg:verification} summarizes the overall process.

\begin{algorithm}[H]
  {\bf Data:} The set of matrices $\cA$, \\
  the maximum number of extreme rays in the cone $\numrays$ \\
  {\bf Result:} $R$ satisfying Proposition \ref{prop:feas} if found, else $False$ \smallskip \\
  {\bf Procedure:} \\
  $\cKRI, \cKHO, R^{(0)} = $ Initialize$(\cA, \numrays)$ \\
  $k = 0$ \\
  {\bf while }{$True$}: \\
  \hspace*{5mm} $w, \{w_i,P_i\} = $ Solution to \eqref{eq:opt} with $R = R^{(k)}$ \\
  \hspace*{5mm} {\bf if }{$w < 0$}: \\
  \hspace*{10mm} return $R^{(k)}$ \\
  \hspace*{5mm} $\delta R = $ Solution to \eqref{eq:opt_small_delta} with $R = R^{(k)}$ \\
  \hspace*{5mm} {\bf if }{$||\delta R|| < \epsilon$}: \\
  \hspace*{10mm} return $False$ \\
  \hspace*{5mm} $R^{(k+1)} = R^{(k)} + \delta R$ \\
  \hspace*{5mm} $k := k + 1$ \\
  \caption{$\cK$-cooperativity Verification Algorithm \label{alg:verification}}
\end{algorithm}

Although this algorithm is not guaranteed to be able to verify $\cK$-cooperativity in general,
if it terminates successfully then the system is guaranteed to be $\cK$-cooperative.
It also appears to be effective in practice,
as we demonstrate in the Examples section.
By modifying existing rays, instead of continuously adding new ones,
this algorithm allows us to
reuse cones found as starting points
for modified problems.
This is crucial for the synthesis problem.

\section{Synthesizing $\cK$-cooperativity} \label{sec:control_design}

\subsection{Problem Formulation}
We next consider the design problem:
how do we find system parameters that make the dynamics $\cK$-cooperative.

We use the following formulation:
\begin{equation}
\dot{x} = f_\parameters(x) = f(x) + \left(\sum_{j = 1}^{\numnonfixed} \parameters_j \designmatrix_j\right)x
\label{eq:synthesis}
\end{equation}
Where $\designmatrix_j$ are fixed matrices, and
$\parameters_j$ are scalars we are free to choose.

Many control design problems can be expressed in this formulation. 
For example, consider a linear system with state feedback $\dot{x} = (A + BF)x$.
Assume it has $n$ states and 1 input.
Then, take $\numnonfixed = n$ and $\designmatrix_j = \outerproduct{B}{e_j}$,
where $e_j$ is a $n$ dimensional vector with a 1 in position $j$ and 0 in all other places.
Each $\parameters_j$ corresponds to $F_j$, the $j^{th}$ element of $F$.

To find parameters that make the dynamics $\cK$-cooperative we need to 
search for a contracting cone while also adapting $\parameters$.
As discussed in Section \ref{subsec:conical_relax},
to be able to test $\cK$-cooperativity with a finite number of computations,
we require a conical relaxation $\cA(\parameters)$,
which now depends on $\parameters$.
For our formulation, $\cA(\parameters)$ has a simple parameterization.
We start with a conical relaxation of $f(x)$, $\cA$.
Then, for each matrix $A_i \in \cA$, we define the parameter dependent matrix $A_i(\parameters) \in \cA(\parameters)$ as:
$$ A_i(\parameters) = A_i + \left(\sum_{j = 1}^{\numnonfixed} \parameters_j \designmatrix_j\right)$$ 
It is easy to see that, since $\cA$ satisfies \eqref{eq:conical_relax} for $\J{x}$,
$\cA(\parameters)$ satisfies \eqref{eq:conical_relax} for $\partial f_\parameters(x) = \partial f(x) + \left(\sum_{j = 1}^{\numnonfixed} \parameters_j \designmatrix_j\right)$. 
We can think of $\cA(\cdot)$ as a function that can be evaluated for a given $\parameters$
to return a valid conical relaxation for $\dot{x} = f_\parameters(x)$.
Propositions \ref{prop:feas} and \ref{prop:opt} still hold for a given $\parameters$ if we set $\cA = \cA(\parameters)$ and $f(x) = f_\parameters(x)$,
providing us with a way to test $\cK$-cooperativity and measure distance to $\cK$-cooperativity,
respectively.
We can also analyze robustness through $\cA(\parameters)$
in the same ways as those described in Section \ref{subsec:conical_relax}.

\begin{remark}
The algorithm that follows can also be used for nonlinear controllers where $f_\parameters(x)$
cannot be decomposed in the same way as \eqref{eq:synthesis}.
To make our exposition cleaner we do not consider this here.
\end{remark}

\subsection{The Algorithm}
\label{subsec:synth_algo}
The high-level structure of the algorithm remains unchanged
from Algorithm \ref{alg:verification}.
However, since the matrices in $\cA(\parameters)$ now depend on $\parameters$, 
we can no longer directly apply the necessary conditions of Section \ref{subsec:ver_necessary}
and do not have obvious candidates for $\primalvector_i$ and $\dualvector_i$.
We instead set all $\primalvector_i$ to a
pre-selected fixed $\primalvector$,
and all $\dualvector_i$ to a
pre-selected fixed $\dualvector$.
This replaces \eqref{eq:geo_assump} with $\cKR(\primalvector) \subset \cK \subset \cKH(\dualvector^T)$,
constraining our search to cones that have 
$\primalvector \in \interior(\cK)$ and $\dualvector \in \interior(\dual{\cK})$.

The cone is initialized in the same way as in Section \ref{subsec:ver_algo} but with $\cKR(\primalvector)$ in place of $\cKRI$ and $\cKH(\dualvector^T)$ in place of $\cKHO$.

We next consider the necessary modifications to
the two optimization steps
corresponding to \eqref{eq:opt} and \eqref{eq:opt_small_delta}.
When incorporating $\parameters$
to our optimization problems we have two choices:
we can either treat $\parameters$ as fixed
when solving the first optimization step
and update $\parameters$ only through solutions to
the second optimization step
(as we do with $R$),
or we can treat $\parameters$ as a variable
directly from the first optimization step
(as we do with $\{\widening_i, \posmatrix_i\}$).
Here we chose to implement the former
because it generalizes to nonlinear formulations
of $\cA(\parameters)$ 
and allows us to bias the values of $\parameters$.
When treating $\parameters$ in this way
we also have to provide initial values $\parameters^{(0)}$ to the algorithm.

Optimization problem \eqref{eq:opt} is replaced by:
\begin{alignat}{2}
   &\!\min_{\{\widening_i,\posmatrix_i\}}\,&\quad& \widening \label{eq:opt_nonfixed_fixed} \\
   &\text{subject to, $\forall\,i \in |\cA(\parameters)|$:}& &\widening \geq \widening_i,\, \offdiag(\posmatrix_i) \geq 0,\nonumber \\
   & & & \left[A_i(\parameters) + \widening_i \outerproduct{\primalvector}{\dualvector}\right]R = R\posmatrix_i \nonumber
\end{alignat}

To obtain an optimization problem analogous to \eqref{eq:opt_small_delta}
we repeat the approximation procedure presented in Section \ref{subsec:ver_algo}, this time applied to \eqref{eq:opt_nonfixed_fixed}.
The only difference is in having to include the
derivative with respect to $\parameters$
when computing the total derivative of the equality constraints.
The following optimization problem is obtained:
\begin{alignat}{2}
   &\!\min_{\delta R,\, \delta \parameters,\,M,\, \{\delta\widening_i,\delta\posmatrix_i\}}\,&\>\,& \bar{\widening} \label{eq:opt_nonfixed_small_delta} \\
   &\text{subject to:}&&
   H_{outer}(R+\delta R) > 0,\, (R+\delta R) = R'M, \nonumber \\
   &&& 1^T M = 1^T,\, \offdiag(M) \geq 0, \nonumber \\
   &&& (\delta R,\,\{\delta\widening_i,\delta\posmatrix_i\}) \text{ small}, \nonumber \\  
   &\text{and, $\forall\,i \in |\cA(\parameters)|$:}&& \bar{\widening} \geq \widening_i + \delta\widening_i,\, \offdiag(\posmatrix_i + \delta\posmatrix_i) \geq 0, \nonumber \\
   &&& [A_i(\parameters) + \widening_i \outerproduct{\primalvector}{\dualvector}](\delta R)
   + (\delta \widening_i) \outerproduct{\primalvector}{\dualvector}R 
   \nonumber \\
   &&& \quad + \left(\sum_{j = 1}^{\numnonfixed}(\delta \parameters_j) \designmatrix_j \right) R \nonumber \\
   &&& \quad = (\delta R)\posmatrix_i + R(\delta\posmatrix_i) \nonumber 
\end{alignat}

The overall process is summarized in Algorithm \ref{alg:synthesis}:
\begin{algorithm}[H]
   {\bf Data:} The set of matrix functions $\cA(\cdot)$, \\
   the initial set of parameters $\parameters^{(0)}$, \\
   a vector that will lie inside the cone $\primalvector$, \\
   a vector that will lie inside the dual cone $\dualvector$, \\
   the number of rays in the cone $\numrays$ \\
   {\bf Result:} $R$ and $\parameters$ such that $R$, $\cA(\parameters)$, and $f_\parameters(x)$ satisfy Proposition \ref{prop:feas} if found, else $False$ \smallskip \\
   {\bf Procedure:} \\
   $R^{(0)} = $ Initialize$(\primalvector, \dualvector^T, \numrays)$ \\
   $k = 0$ \\
   {\bf while }{$True$}: \\
   \hspace*{5mm} $w, \{w_i,P_i\} = $ Solution to \eqref{eq:opt_nonfixed_fixed} with $R = R^{(k)},\, \parameters = \parameters^{(k)}$ \\
   \hspace*{5mm} {\bf if }{$w < 0$}: \\
   \hspace*{10mm} return $R^{(k)},\,\parameters^{(k)}$ \\
   \hspace*{5mm} $\delta R, \delta \parameters = $ Solution to \eqref{eq:opt_nonfixed_small_delta} with $R = R^{(k)},\, \parameters = \parameters^{(k)}$ \\
   \hspace*{5mm} {\bf if }{$||\delta R|| < \epsilon_R$ and
   $||\delta \parameters|| < \epsilon_{\parameters}$}: \\
   \hspace*{10mm} return $False$ \\
   \hspace*{5mm} $R^{(k+1)} = R^{(k)} + \delta R$ \\
   \hspace*{5mm} $\parameters^{(k+1)} = \parameters^{(k)} + \delta \parameters$ \\
   \hspace*{5mm} $k := k + 1$ \\
   \caption{$\cK$-cooperativity Synthesis Algorithm \label{alg:synthesis}}
 \end{algorithm}

\section{Examples} \label{sec:examples}

\subsection{Electrical Switch} \label{example:electrical_switch}
Taking inspiration from the bistable and oscillatory circuits analyzed in 
\cite{miranda-villatoro_differentially_2018,miranda-villatoro_dissipativity_2019},
we consider the electrical circuit with three states shown in \figref{fig:electrical_circuit} (left).
The dynamics of this circuit are given by:
 \begin{equation}
 \begin{split}
   L_1 \dot{x}_1 &= -R_1 x_1 + x_2 - x_3 \\
   C_2 \dot{x}_2 &= -x_1 - x_2/R_2 \\
   C_3 \dot{x}_3 &= x_1 - f(x_3)
 \end{split}
 \label{eq:electrical_switch}
\end{equation}

Here $x_1$ represents the current across the inductor, and
$x_2$ and $x_3$ represent the voltages across the capacitors.
The scalar function $f(\cdot)$ represents a nonlinear resistor.
We take its derivative, $f'$, to be bounded by $f'_l \leq f' \leq f'_u$.

For different choices of parameters, this system exhibits a range of different qualitative behaviors. 
We fix \electricalvaluescommon{} and let \electricalnonlineardefault{},
shown in \figref{fig:electrical_circuit} (right),
such that \electricalnonlinearbounds{}.
Then, for $R_1 = \electricalresistanceosc$, simulations show the existence of a limit cycle,
while for $R_1 = \electricalresistancebi$ the system appears to be bistable.
This is illustrated in \figref{fig:electrical_circuit_analysis} (left).

\begin{figure}[tbp]
   \begin{minipage}{.49\linewidth}
      \begin{center}
      \begin{circuitikz}[scale=.5,transform shape]
         \tikzstyle{every node}=[font=\Large]
         \draw (0,0)
         to[short] (0,2)
         to[R=$R_2$] (0,4) 
         to[short] (0,5)
         to[short] (2,5)
         to[short] (2,4)
         to[C=$C_2$, v=$x_2$] (2,2) 
         to[short] (2,0)
         to[short] (0,0);
         \draw (2,5)
         to[L=$L_1$, i=$x_1$] (6,5)
         to[R=$R_1$] (6,3)
         to[short] (4,3)
         to[short] (4,2)
         to[C=$C_3$, v=$x_3$] (4,0)
         to[short] (2,0);
         \draw (6,3)
         to[ageneric, i=$f(x_3)$] (6,0)
         to[short] (4,0);
         \draw (6,3)
         to[R, bipoles/length=.6cm] (6,0);
         \draw (2,0)
         node[ground]{};
      \end{circuitikz}
      \end{center}
   \end{minipage}
   \begin{minipage}{.49\linewidth}
      \includegraphics[width=\linewidth]{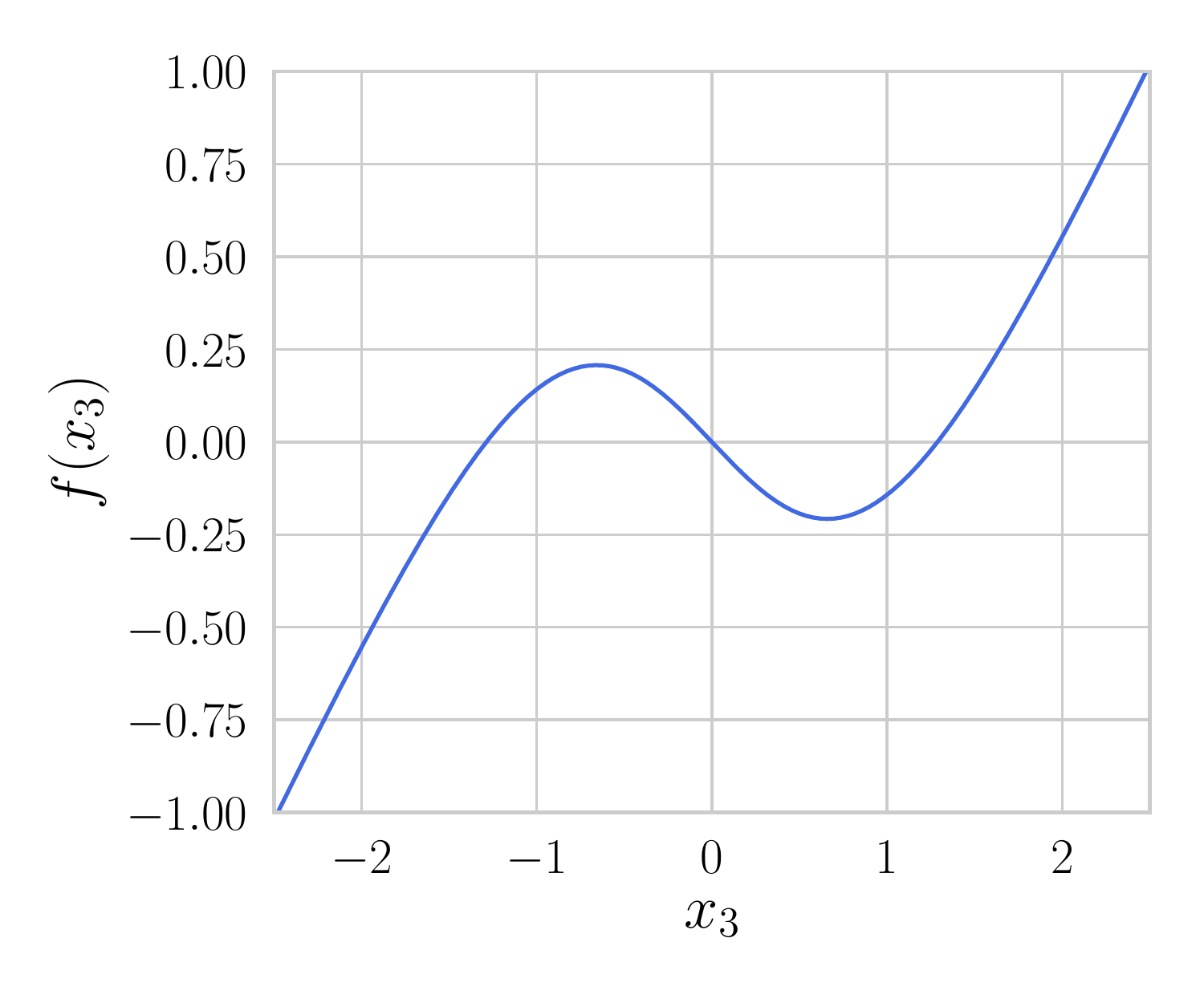}
   \end{minipage}
   \caption{\textbf{Left:} Circuit used in Section \ref{example:electrical_switch}. \\
   \textbf{Right:} Plot of static nonlinearity $f(x_3)$.
   }
   \label{fig:electrical_circuit}
 \end{figure}
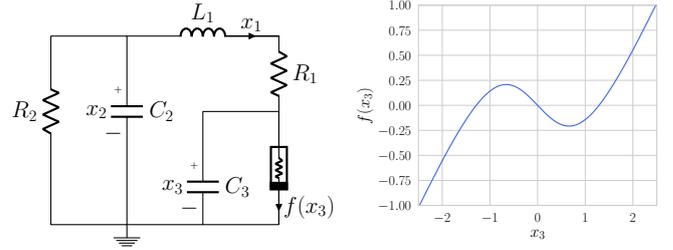

Algorithm \ref{alg:verification} finds a contracting cone for the case where $R_1 = \electricalresistancebi$, proving that it is $\cK$-cooperative.
Since the chosen nonlinear resistor guarantees bounded
trajectories and our system has two stable equilibrium points,
this in turn proves that it is bistable.
To further investigate this scenario,
we incorporate robustness considerations:
(i) We add matrix $[0, 0, 1]^T[0, 0, 1]$ to our conical relaxation.
This removes the lower bound on $f'$ ($f(\cdot)$ must still be Lipschitz).
(ii) We allow all passive components (the resistances, capacitors, and inductor)
other than $R_1$
to vary by $\pm \electricaltolerance$, generating matrices for each possible scenario.
The robust cone found is shown in \figref{fig:electrical_circuit_analysis} (right).

We now attempt to use Algorithm \ref{alg:synthesis} to find an $R_1$ for which we can prove that the system is $\cK$-cooperative even if
the passive components are allowed to vary by $\pm \electricalnewtolerance{}$.
We initialize the nominal values at those of the oscillatory scenario from \figref{fig:electrical_circuit}.
We still include matrix $[0, 0, 1]^T[0, 0, 1]$ in our conical relaxation, removing the lower bound on $f'$.
As such we set $\primalvector = \dualvector = [0, 0, 1]^T$, $\parameters = R_1$, $\parameters^{(0)} = \electricalresistanceosc$.
We find the cone shown in \figref{fig:electrical_circuit_synthesis} (left),
which is contracted for our system
provided $R_1 = \electricalnewresistancebi{}$ and all passive components (other than $R_1$) are within $\pm \electricalnewtolerance{}$ of their nominal values.
We can therefore conclude that, for the above range of component values,
the system will be bistable as long as trajectories are bounded
and the equilibrium at $0$ is unstable.
Some possible example trajectories are shown in \figref{fig:electrical_circuit_synthesis} (right).

\begin{figure}[tbp]
   \begin{minipage}{.49\linewidth}
      \includegraphics[width=\linewidth]{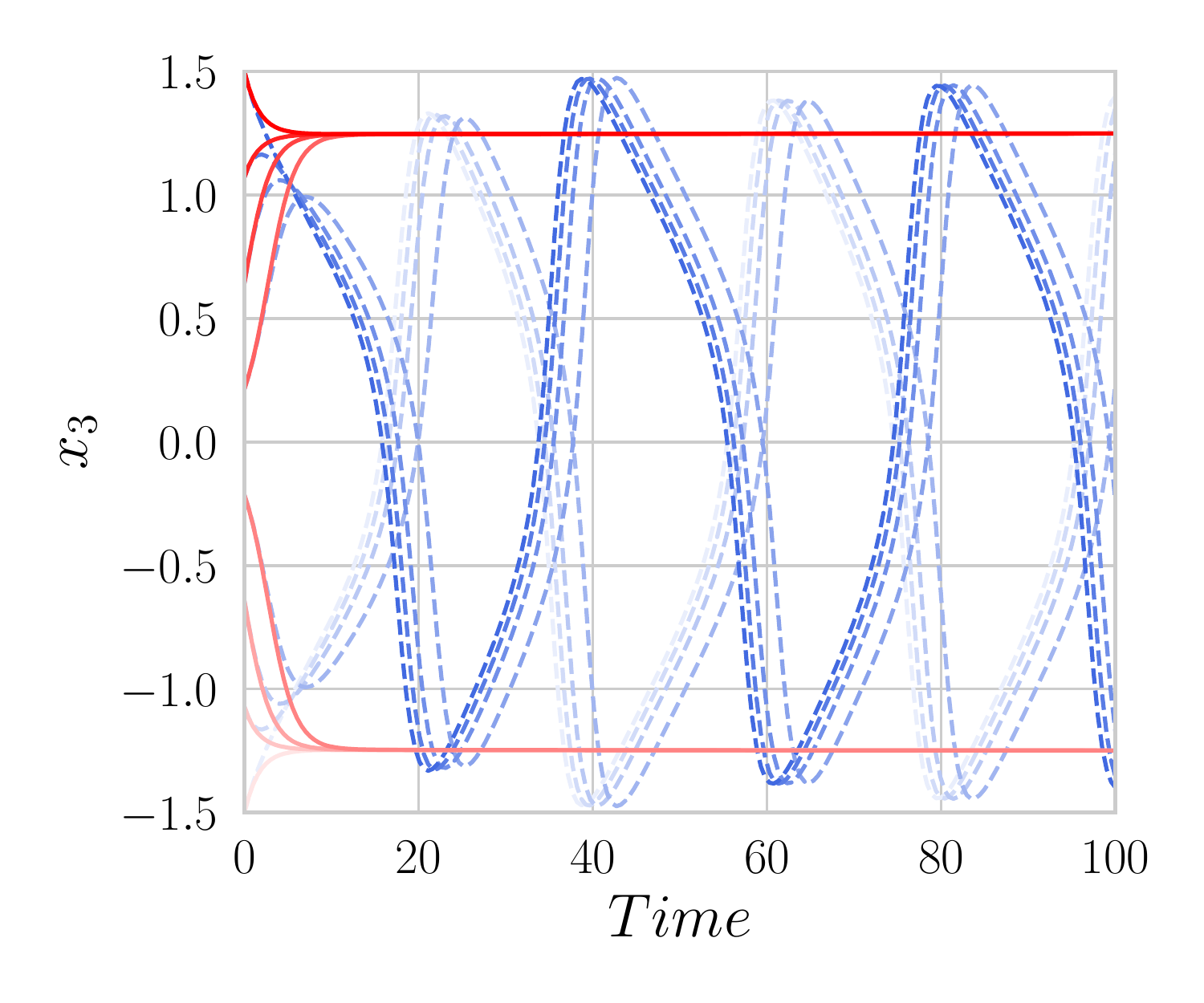}
   \end{minipage}
   \begin{minipage}{.49\linewidth}
      \includegraphics[width=\linewidth]{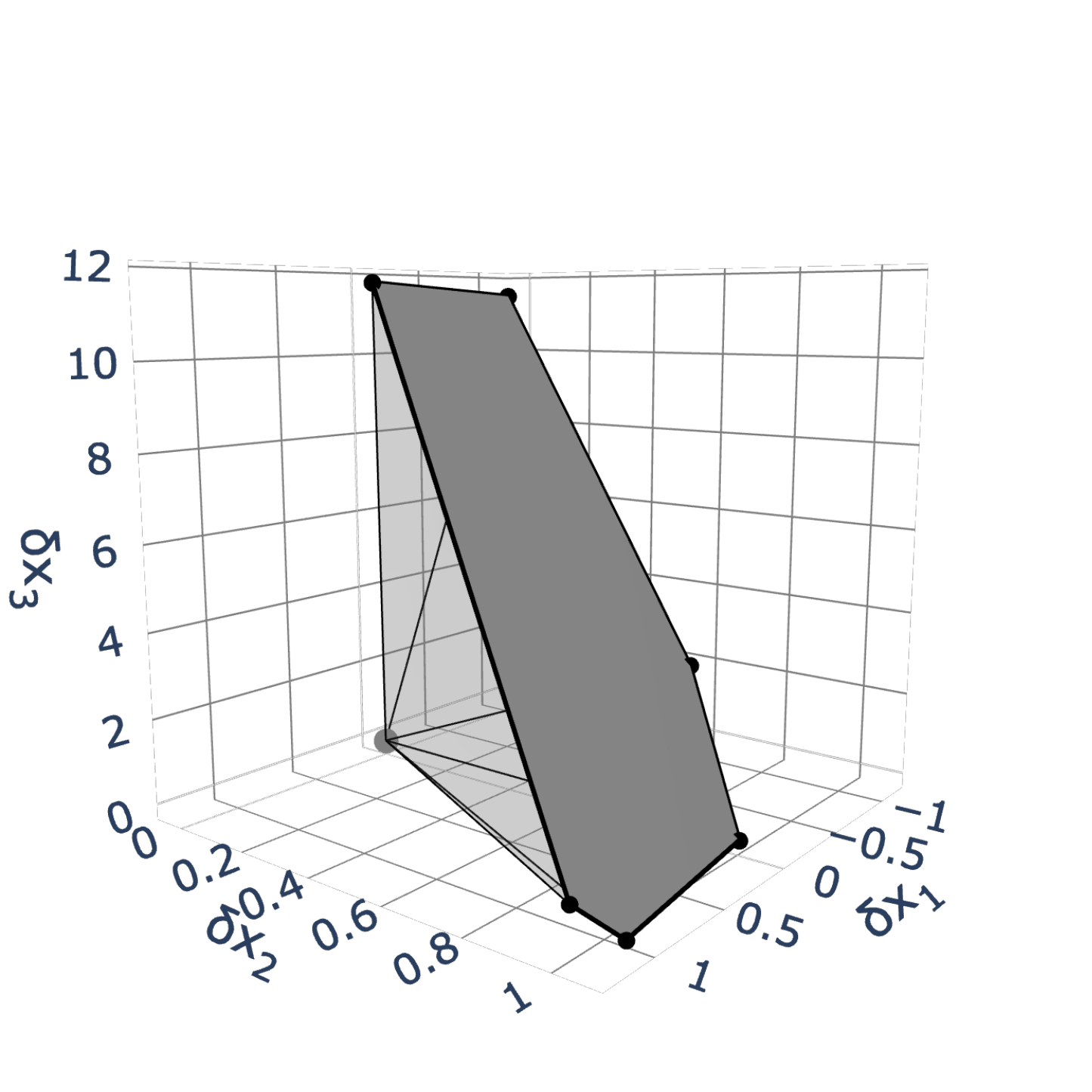}
   \end{minipage}
   \caption{
   \textbf{Left:} Example trajectories of \eqref{eq:electrical_switch} with $R_1 = \electricalresistanceosc$ (blue and dashed) and $R_1 = \electricalresistancebi$ (red). \\
   \textbf{Right:} Cone found for \eqref{eq:electrical_switch} using Algorithm \ref{alg:verification} ($R_1 = \electricalresistancebi$, component tolerance $\pm \electricaltolerance{}$).
   }
   \label{fig:electrical_circuit_analysis}
 \end{figure}

\begin{figure}[tbp]
   \begin{minipage}{.49\linewidth}
      \includegraphics[width=\linewidth]{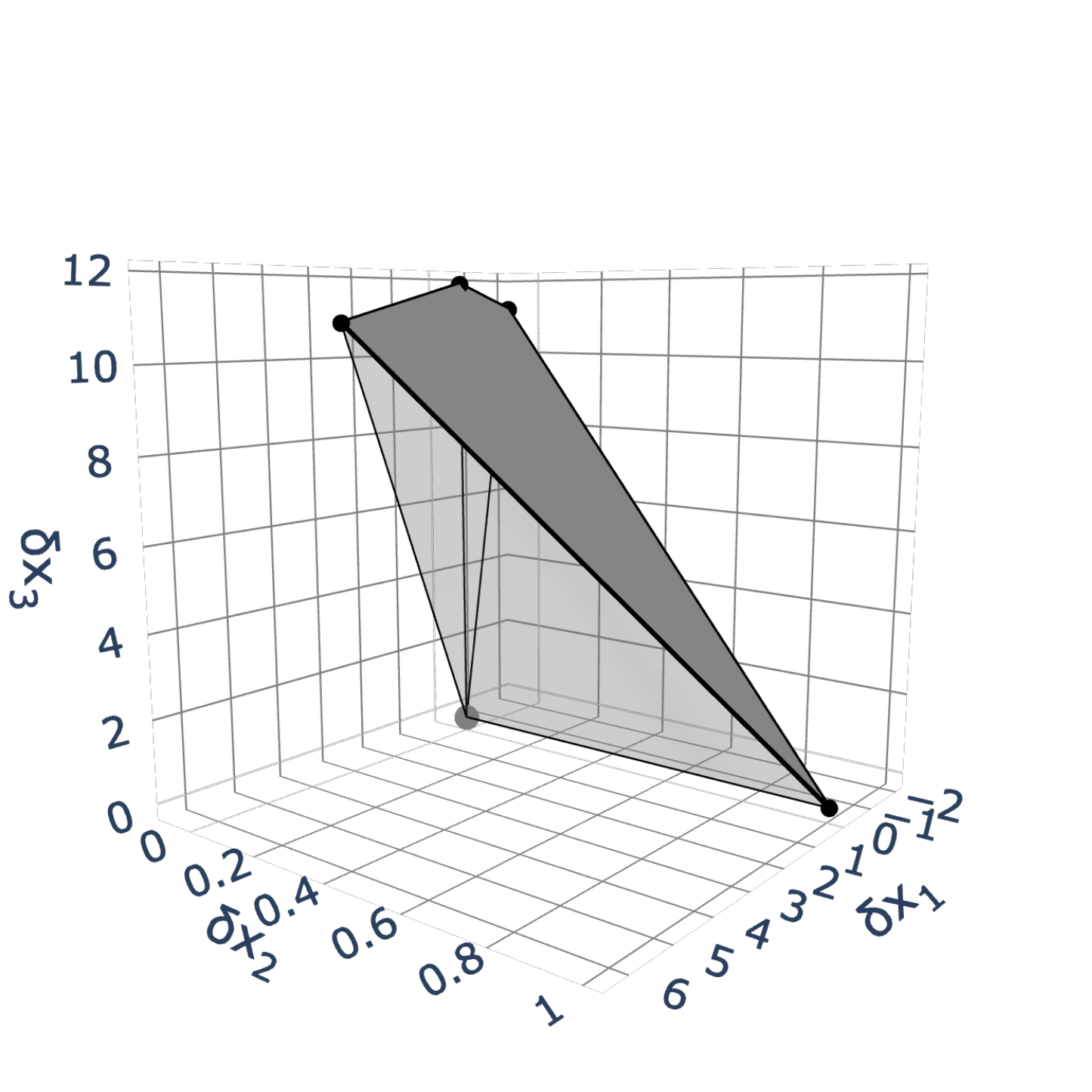}
   \end{minipage}
   \begin{minipage}{.49\linewidth}
      \includegraphics[width=\linewidth]{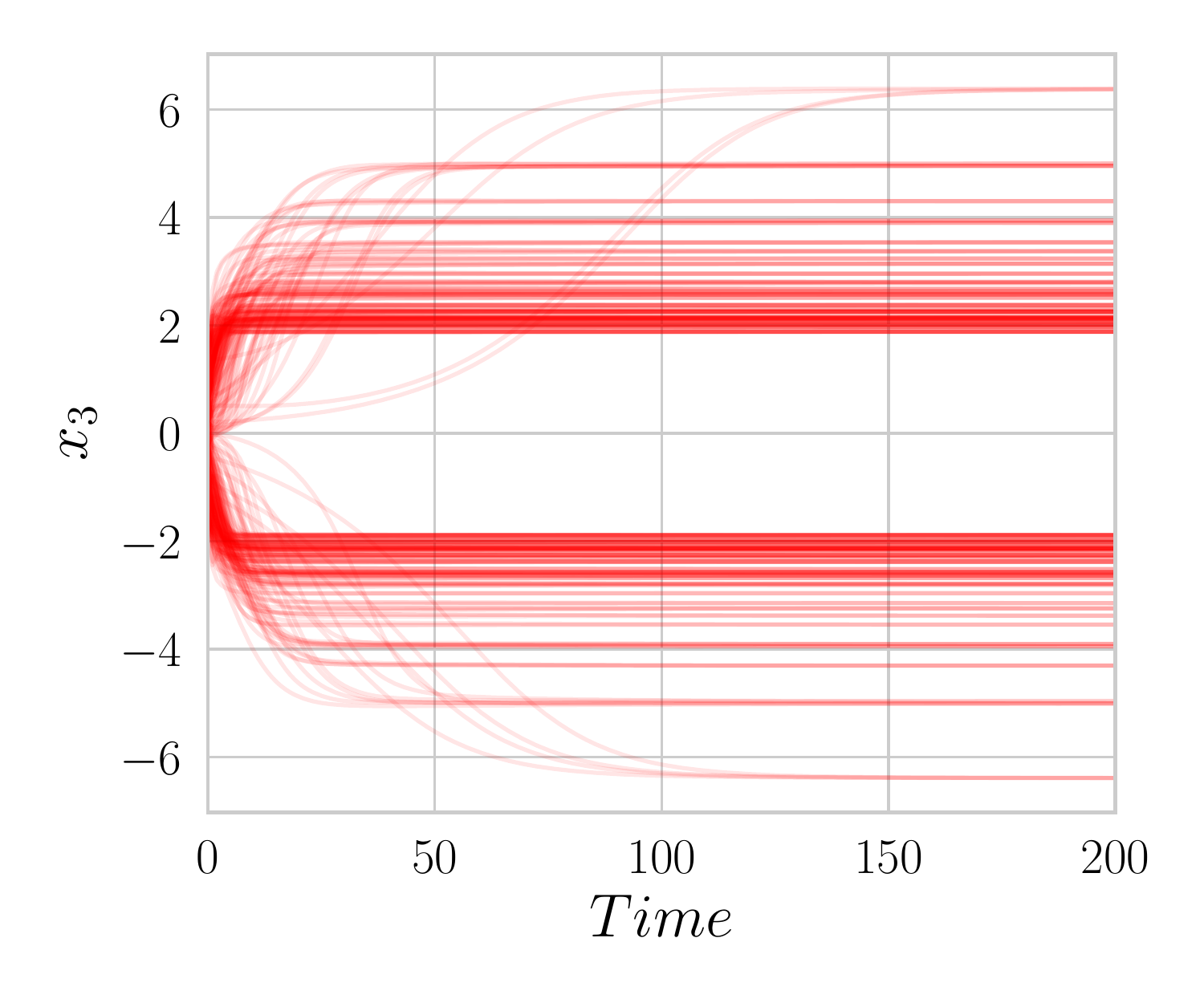}
   \end{minipage}
   \caption{
   \textbf{Left:} Cone found for \eqref{eq:electrical_switch} using Algorithm \ref{alg:synthesis} ($R_1^{(0)} = \electricalresistanceosc$, final $R_1 = \electricalnewresistancebi{}$, component tolerance $\pm \electricalnewtolerance{}$). \\
   \textbf{Right:} Example trajectories of \eqref{eq:electrical_switch}
   with different component values
   for which the cone on the left guarantees bistability.
   }
   \label{fig:electrical_circuit_synthesis}
 \end{figure}

\subsection{First Order Consensus} \label{example:consensus}
We consider dynamics of the form:
\begin{equation}
   \dot{x_i} = \sum_{j=1}^N f_{ij}(x_j-x_i) \qquad 0 < i \leq N,
   \label{eq:cons_nonlinear} 
\end{equation}
where each agent $x_i$ represents a simple integrator 
driven by the weighted differences with its neighboring agents,
characterized by functions
$f_{ij}: \mathbb{R} \to \mathbb{R}$ such that $f_{ij}(0) = 0$.
We say that the agents reach consensus when $x_1 = \cdots = x_N$.

We can use $\cK$-cooperativity to study consensus of nonlinear or uncertain systems:
the trajectories of a strictly $\cK$-cooperative system with $1_N \in \interior(\cK)$
converge to consensus.

Here, we revisit the consensus example from Section V.A of \cite{kousoulidis_finding_2019},
where we considered a network of 5 agents with the topology shown in \figref{fig:cons_graph}.
The black edges represent linear connections with weights normalized to 1,
the blue edges represent nonlinear connections with slopes restricted to $-1 \leq f_{15}' \leq 1$ and $-1 \leq f_{42}' \leq 1$, and the red edge represents a linear connection with weight set to $k$.

We initially fix $k = 1$ and attempt to use Algorithm \ref{alg:verification}.
We successfully find a contracting cone with \consensusanalnumraysnew{} rays.
For comparison,
when we use the older Algorithm from \cite{kousoulidis_finding_2019} on this problem,
we obtain a cone composed of \consensusanalnumraysold{} rays.

Next, we attempt to find a $k$ such that the system reaches consensus for any
$-\consensussynampl \leq f_{15}' \leq \consensussynampl$ and $-\consensussynampl \leq f_{42}' \leq \consensussynampl$.
For this we use Algorithm \ref{alg:synthesis} with $\primalvector, \dualvector = 1_N$, $c = k$ and $c^{(0)} = 1$.
We find a contracting cone with $k=\consensussyngain$.
An example of a random trajectory when $f_{15}(x), f_{42}(x) = -\consensussynampl \sin(x)$
with $k = 1$ and $k = \consensussyngain$ is shown in \figref{fig:cons_traj}. 
Since $-\consensussynampl \leq f_{15}' \leq \consensussynampl$ and $-\consensussynampl \leq f_{42}' \leq \consensussynampl$,
we can only guarantee consensus for the case where $k = \consensussyngain$,
and plot the trajectory when $k = 1$ for comparison.

\begin{figure}[tbp]
	\centering
	\begin{minipage}{.49\linewidth}
		\centering
		\begin{tikzpicture}[x=0.7cm,y=0.9cm]
		\SetUpEdge[lw         = 1.5pt,
		color      = black]
		\GraphInit[vstyle=Normal] 
		\SetGraphUnit{1.4}
		\tikzset{every node/.style={fill=yellow}}
		\tikzset{VertexStyle/.append  style={fill}}
		\Vertex{1}
		\NOEA(1){2}
		\SOEA(1){3}
		\EA(2){4}
		\EA(3){5}
		\tikzset{EdgeStyle/.style={->}}
		\Edge(2)(1)
		\Edge(4)(2)
		\Edge(3)(2)
		\Edge(1)(3)
      \Edge(1)(5)
      \tikzset{EdgeStyle/.style={->, color=red}}
      \Edge[label=$k$, labelcolor=none, labelstyle={right}](5)(4)
		\tikzset{EdgeStyle/.style={->, color=blue, bend right}}
		\Edge[label=$f_{15}$, labelcolor=none, labelstyle={above, pos=.4}](5)(1)
		\tikzset{EdgeStyle/.style={->, color=blue, bend right}}
		\Edge[label=$f_{42}$, labelcolor=none, labelstyle={below}](2)(4)
		\end{tikzpicture}
	\end{minipage}
\begin{minipage}{.49\linewidth}
	\includegraphics[width=\linewidth]{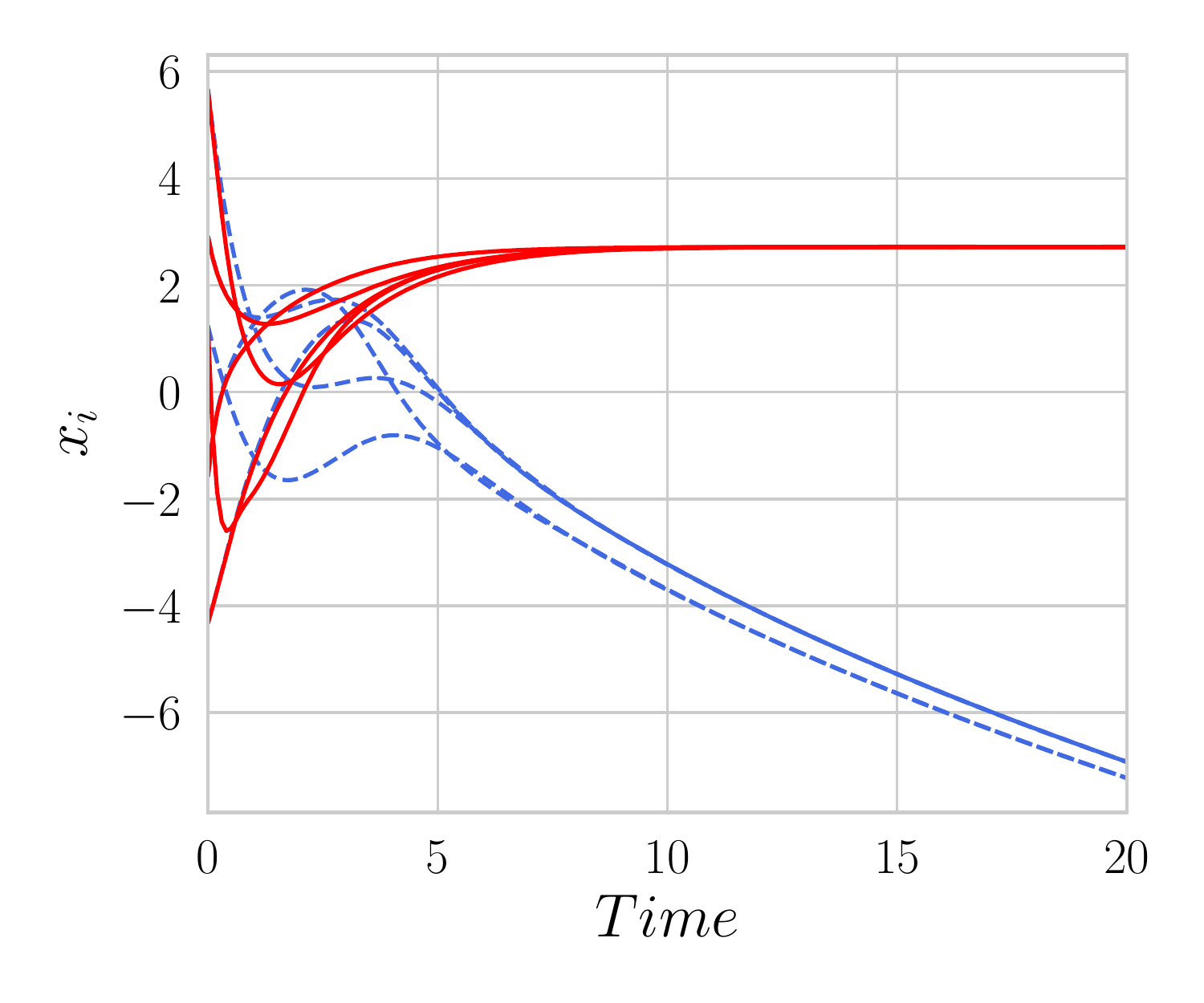}
	\end{minipage}
	\caption{\textbf{Left:} Consensus topology.
   \textbf{Right:} Comparison of a random trajectory of the consensus system with
   $f_{15}(x), f_{42}(x) = -\consensussynampl \sin(x)$
   between when $k = 1$ (blue) and when $k = \consensussyngain$ (red).
	}
	\label{fig:cons_graph}
	\label{fig:cons_traj}
\end{figure}

\subsection{Spring-Damper: Lyapunov Analysis} \label{example:spring_damper}
As a final example, we show some initial results on how our algorithm can also
be adapted to find Lyapunov functions and design robust stabilizing controllers for linear time-varying systems described by 
polytopic Linear Differential Inclusions (LDIs), $\dot{x} = A(t) x$ with $A(t) \in \convexhull(\cA)$;
where we define $\convexhull(\cA)$ as:
$$\convexhull(\cA) = \left\{ A : A = \sum_{i=1}^k \posvector_i A_i , \,\posvector_i \geq 0, \sum_{i=1}^k \posvector_i = 1 \right\}$$

We can do this by finding a bounded polytope $\cC$ that is contracted by $\cA$. That is a $\cC$ such that for every $A_i \in \cA$, if $\dot{x} = A_i x$ then
\begin{equation}
  x(0) \in \cC \implies x(t) \in \interior(\cC) \text{ for all } t \geq 0.
  \label{eq:contract_polytope}
\end{equation}
The Minkowski functional of $\cC$ is then a Lyapunov function for the LDI
(\cite{blanchini_set_1999}).

To use our algorithm to find bounded polytopes instead of polyhedral cones,
we begin by transforming our $n$ dimensional matrices $A_i \in \cA$ into extended $n+1$ dimensional matrices $\hat{A}_i \in \hat{\cA}$:
\begin{equation}
\hat{A}_i =
\begin{bmatrix}
   0 & 0_n^T \\
   0_n & A_i
\end{bmatrix}
\end{equation}
And we set all $\primalvector_i,\,\dualvector_i = [1, 0_n]^T$.

We then apply Algorithm \ref{alg:verification} or \ref{alg:synthesis}
as normal, obtaining a $n+1$ dimensional $\cKR(R)$ that satisfies 
\eqref{eq:strict-subtang} for each $\hat{A}_i \in \hat{\cA}$ and,
in the case of Algorithm \ref{alg:synthesis},
a corresponding vector of parameters $\parameters$.

We map a $\cKR(R)$ that satisfies \eqref{eq:strict-subtang} to a $\cC$ that satisfies \eqref{eq:contract_polytope}
by taking a slice of the cone about 
the hyperplane $\innerproduct{[1,0_n]}{x} = 1$.
Numerically, if each column of $R$ is scaled such that the first element is 1,
this is equivalent to removing the first row of $R$,
with each column of the matrix obtained representing
a vertex of $\cC$.


We apply this to the following time varying spring-damper system:
\begin{equation}
\begin{bmatrix}
   \dot{p} \\
   \dot{v}
\end{bmatrix}
=
\begin{bmatrix}
   0 & 1 \\
   -(\phi(t) + k_p) & -2
\end{bmatrix}
\begin{bmatrix}
   p \\
   v
\end{bmatrix}
\label{eq:spring_damper}
\end{equation}
Where $\phi(t)$ represents an uncertain (and potentially negative) spring constant and $k_p$ represents a proportional position feedback gain.

We initially consider the analysis problem by setting 
$\springdamperanalmin \leq \phi(t) \leq \springdamperanalmax$
and $k_p = 0$.
We use Algorithm \ref{alg:verification} to obtain the contracting polytope shown in \figref{fig:spring_damper_analysis} (left).
We then attempt the synthesis problem, setting $\springdampersynmin \leq \phi(t) \leq \springdampersynmax$ and letting $c = k_p$ be a free parameter
initialized to 0.
Using Algorithm \ref{alg:synthesis},
we simultaneously obtain a stabilizing gain ($k_p = \springdampersyninitialkp$) and a contracting polytope,
as shown in \figref{fig:spring_damper_synthesis} (right).
It is worth noting that attempting to find a Quadratic Lyapunov function for the LDI
using linear matrix inequalities fails when
$\springdampersynmin \leq \phi(t) \leq \springdampersynmax$
and $k_p = \springdampersyninitialkp$.

\begin{figure}[tbp]
   \begin{minipage}{.49\linewidth}
      \includegraphics[width=\linewidth]{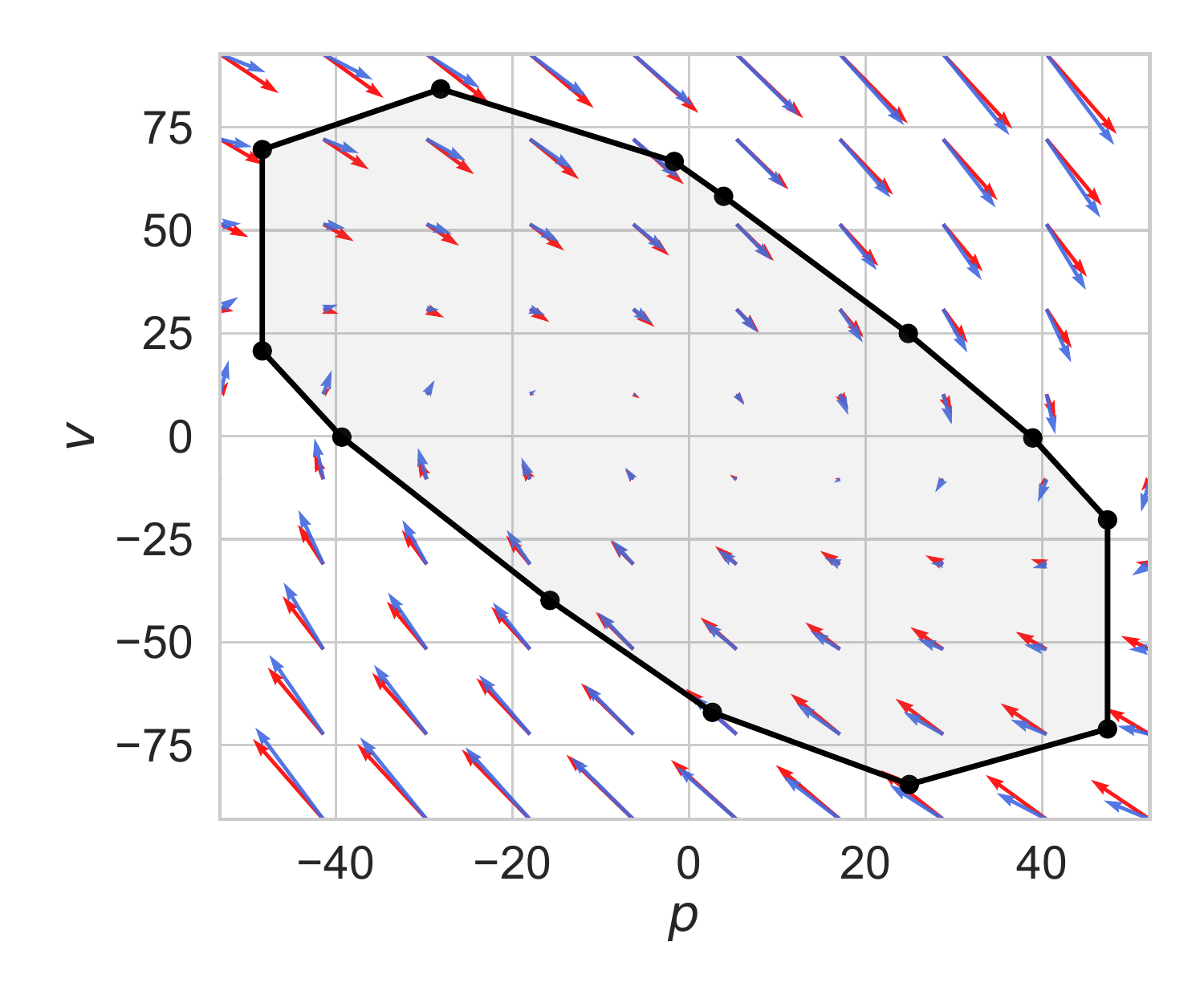}
   \end{minipage}
   \begin{minipage}{.49\linewidth}
      \includegraphics[width=\linewidth]{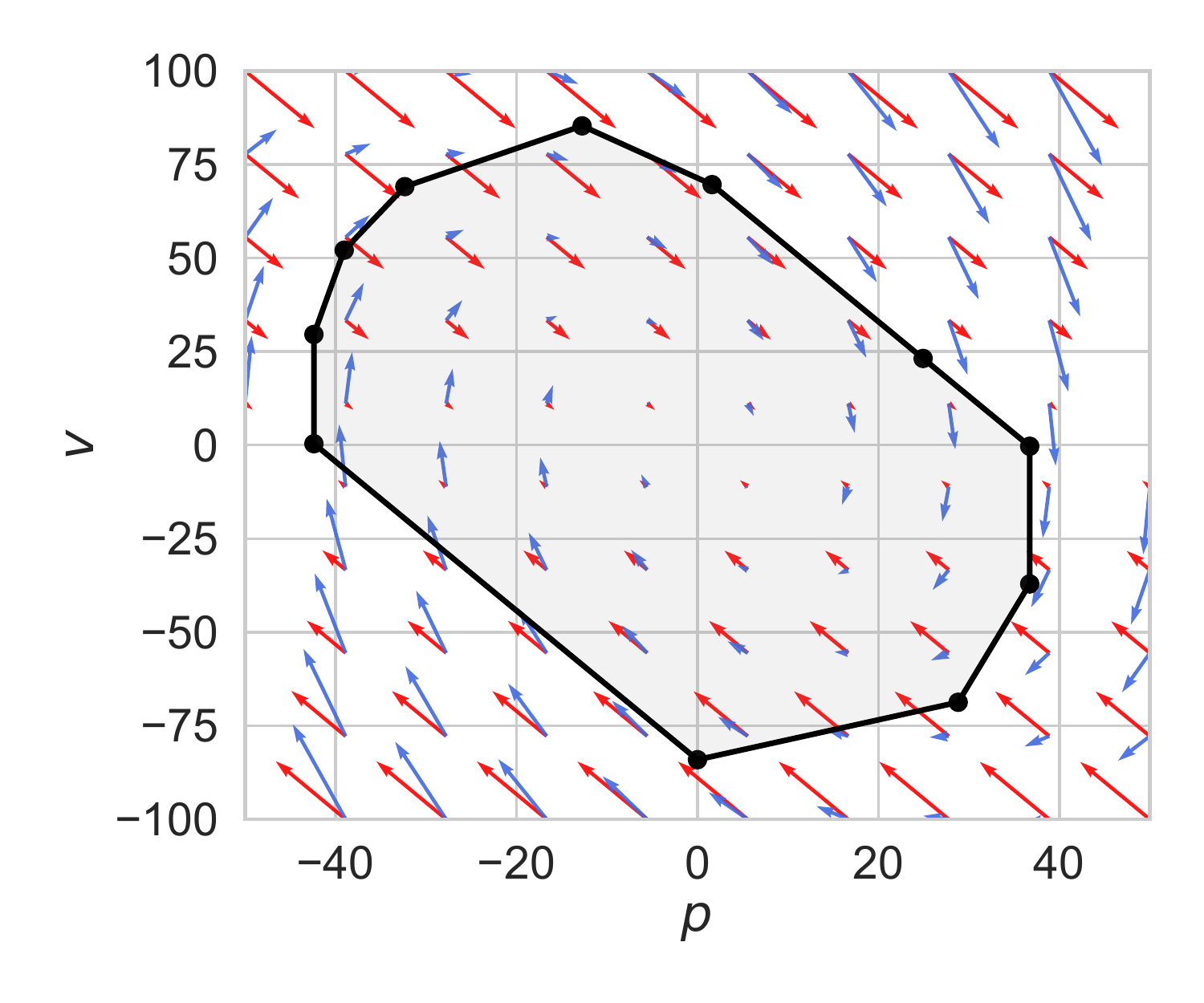}
   \end{minipage}
   \caption{
      Contracting Polytopes found for \eqref{eq:spring_damper}
      and the vector fields 
      generated by each of the two $A_i \in \cA$.\\
      \textbf{Left:} Analysis problem with $\springdamperanalmin \leq \phi(t) \leq \springdamperanalmax$
      and $k_p = 0$.\\
      \textbf{Right:} Synthesis problem with
      $\springdampersynmin \leq \phi(t) \leq \springdampersynmax$
      and $k_p^{(0)} = 0$.
      Final $k_p = \springdampersyninitialkp$.
   }
   \label{fig:spring_damper_analysis}
   \label{fig:spring_damper_synthesis}
\end{figure}

In this example we have shown a new method for computing polyhedral Lyapunov functions
(\cite{blanchini_set-theoretic_2015}).
We leave a thorough investigation of this application area,
including its differential version
(\cite{forni_differential_2014}),
as future work.

\section{Conclusions}
We presented a novel algorithm for constructing contracting polyhedral cones.
The algorithm is based on Linear Programming and keeps the number of vectors in the representation of the cone fixed.
This enabled us to verify and 
synthesize $\cK$-cooperative systems without fixing a cone a priori.
Building on differential positivity, $\cK$-cooperative systems
capture more behaviors than traditional differential analysis
approaches, which we illustrated
by providing examples of robust analysis and synthesis for
bistable and multi-agent systems.

\begin{ack}
The authors wish to thank I. Cirillo and F. Miranda for
useful comments and suggestions to the manuscript.
\end{ack}

\bibliography{20_IFAC}

\end{document}